\documentclass[11pt]{amsart}
\usepackage{amsmath,amssymb,enumerate,bbm,enumerate}
\numberwithin{equation}{section}
\usepackage{xcolor} 
\usepackage{graphicx}
\makeatletter

\newtheorem{theorem}{Theorem}[section]

\newtheorem{theo}[theorem]{Theorem}
\newtheorem{pro}[theorem]{Proposition}
\newtheorem{thm}[theorem]{Theorem}

\theoremstyle{definition}

\newtheorem{defi}[theorem]{Definition}

\theoremstyle{remark}
\newtheorem{remark}[theorem]{Remark}

 \theoremstyle{plain}
\newtheorem*{namedthm}{\namedthmname}
\newcounter{namedthm}

\makeatletter
\newenvironment{named}[1]
  {\def\namedthmname{#1}%
   \refstepcounter{namedthm}%
   \namedthm\def\@currentlabel{#1}}
  {\endnamedthm}

 \newcommand{\R}{\mathbb R}

 \newcommand{\e}{\varepsilon}
 \newcommand{\f}{\varphi}
 
 \newcommand{\p}{\psi}

 \newcommand \mrm{\mathrm}

\author{Vincent Guedj}
\address{Vincent Guedj, Institut de Math\'ematiques de Toulouse  \\ Universit\'e de Toulouse, CNRS \\
UPS  IMT\\
118 route de Narbonne \\
F-31062 Toulouse cedex 09}
\email{\href{mailto:vincent.guedj@math.univ-toulouse.fr}{vincent.guedj@math.univ-toulouse.fr}}
\urladdr{\href{https://www.math.univ-toulouse.fr/~guedj}{https://www.math.univ-toulouse.fr/~guedj/}}

\author{Chinh H. Lu}
\address{Chinh H. Lu, Laboratoire de Math\'ematiques d'Orsay,
 Univ. Paris-Sud,
 CNRS, Universit\'e Paris-Saclay,
  91405 Orsay, France}
\email{\href{mailto:hoang-chinh.lu@math.u-psud.fr}{hoang-chinh.lu@math.u-psud.fr}}
\urladdr{\href{https://www.math.u-psud.fr/~lu/}{https://www.math.u-psud.fr/~lu/}}

\author{Ahmed Zeriahi}
\address{Ahmed Zeriahi, Institut de Math\'ematiques de Toulouse  \\ Universit\'e de Toulouse, CNRS \\
UPS  IMT\\
118 route de Narbonne \\
F-31062 Toulouse cedex 09}
\email{\href{mailto:ahmed.zeriahi@math.univ-toulouse.fr}{ahmed.zeriahi@math.univ-toulouse.fr}}
\urladdr{\href{https://www.math.univ-toulouse.fr/~zeriahi/}{https://www.math.univ-toulouse.fr/~zeriahi/}}

\thanks{The authors are partially supported by the ANR project GRACK}


\title[Stability of complex Monge-Amp\`ere flows]{Stability of solutions to   complex Monge-Amp\`ere flows}
\date{\today}

\usepackage{hyperref}
\hypersetup{
    bookmarks=true,         
    unicode=false,          
    pdftoolbar=true,        
    pdfmenubar=true,        
    pdffitwindow=false,     
    pdfstartview={FitH},    
    pdftitle={Stability of solutions to complex Monge-Ampere flows},    
    pdfauthor={Vincent Guedj,Chinh H. Lu, Ahmed Zeriahi},     
    colorlinks=true,       
   linkcolor=black,          
    citecolor=black,        
    filecolor=black,      
    urlcolor=black}           

\begin{document}
\begin{abstract}
We establish a stability result for elliptic and parabolic complex Monge-Amp\`ere equations on compact K\"ahler manifolds,
which applies in particular to  the K\"ahler-Ricci flow. \\

\noindent {\it Dedicated to Jean-Pierre Demailly on the occasion of his 60th birthday.}
\end{abstract}
 \setcounter{tocdepth}{1}

\maketitle

\tableofcontents

\section*{Introduction}

 Several important problems of K\"ahler geometry  (e.g. finding canonical metrics) necessitate to study the existence and regularity of solutions to certain degenerate complex Monge-Amp\`ere equations.
The fundamental work of Yau \cite{Yau78} guarantees the existence of smooth solutions to a large class of equations.
 It  has been generalized in various interesting directions, providing weak solutions to several degenerate situations.
We refer the reader to 
\cite{GZbook2017} for a recent overview.

In this paper we are interested in the stability properties of solutions to such equations. 
Let $X$ be a compact K\"ahler manifold of dimension $n$, $\theta$ be a K\"ahler form  and $dV$ a volume form on $X$.
Fix $0 \leq f$ a positive density on $X$ and let $\varphi$ be a $\theta$-plurisubharmonic function
solving 
\begin{equation} \label{eq:AY-equation}
{\rm MA}_{\theta}(\varphi) = e^{ \varphi} f d V,
\end{equation}
where   ${\rm MA}_{\theta}(\varphi) = (\theta + dd^c \varphi)^n$ denotes the  
complex Monge-Amp\`ere measure with respect to the form $\theta$.

If $(f_j d V)$ is a sequence of Borel measures converging in total variation to $f d V$ then, it follows from \cite{GZ-stability12} that
the corresponding sequence of solutions $(\varphi_j)$ converges in $L^1(X)$ to  $\varphi$.  
Our aim is to establish a quantitative version of this convergence. 


\smallskip

Our first main result gives a satisfactory answer for $L^p$ densities with respect to Lebesgue measure, $p>1$
(see \cite{Kol03,Blo03,GZ-stability12} for related results).  
 
\begin{named}{Theorem A} 	\label{thm: main thm A} 
Fix $p>1$ and $0\leq f,g\in L^p(X,dV)$. If  $\varphi,\psi$ are  bounded $\theta$-plurisubharmonic  functions on $X$  such that 
        \[
        {\rm MA}_{\theta}(\varphi) =  e^{\varphi} fdV \ ; \   {\rm MA}_{\theta}(\psi) = e^{\psi} gdV,
        \]
        then
        \[
        \Vert \varphi- \psi \Vert_{\infty} \leq   C \|f-g\|_p^{1/n}.  
        \]
    where $C>0$  depends on $p,n,X,\theta,dV$ and uniform bounds on   $\Vert f \Vert_p, \Vert g \Vert_p$.
\end{named}

\vskip.3cm
 
We use some ideas of the proof of \ref{thm: main thm A} to 
establish a stability result for parabolic complex Monge-Amp\`ere flows.
We consider the   equation 
   \begin{equation} \label{eq: PMAE}
  (\omega_t + dd^c \f_t)^n = e^{\dot \f + F (t, \cdot, \f)} f  d V, 
  \end{equation}
 in $X_T := ]0,T[ \times X$,
 where $0<T<+\infty$ and
 \begin{itemize}
 \item  $(\omega_t)_{t \in [0,T]}$ is a smooth family of  K\"ahler forms on $X$;
 \item there exists a fixed K\"ahler form $\theta$ such that $\theta \leq \omega_t$ for all $t$;
 \item $F =F(t,x,r) : \Hat{ X}_T := [0,T] \times X \times \R \rightarrow \mathbb{R}$ is smooth, non-decreasing in the last variable, uniformly Lipschitz in the first and the last variables, i.e. 
$\exists L  > 0$ s.t. $\forall ((t_1,t_2),x,(r_1,r_2)) \in [0,T]^2 \times X \times \R^2, $
\begin{equation} \label{eq:Lip}
  \vert F (t_1,x,r_1) - F (t_2,x,r_2) \vert \leq L \left ( \vert r_1 - r_2 \vert +   \vert t_1 - t_2 \vert \right), 
\end{equation}
 \item $ 0< f\in \mathcal{C}^{\infty}(X,\mathbb{R})$.
 \end{itemize}
 
 Our second main result is the following: 
  
  \begin{named}{Theorem B} \label{thm: main thm B} 
  Fix  $F, G : \Hat{ X}_T := [0,T] \times X \times \R \rightarrow \mathbb{R}$ satisfying the above conditions and fix $p>1$. Assume that $\f : [0,T] \times X \longrightarrow \R$ is a smooth solution to the parabolic equation \eqref{eq: PMAE} with data $(F,f)$ and $\psi :[0,T[ \times X \longrightarrow \R$ is a smooth solution to   \eqref{eq: PMAE} with data $(G,g)$.
  Then 
$$
 \sup_{X_T} \vert \f - \psi\vert  \leq  \sup_X \vert \f_0 - \psi_0\vert + T \, \sup_{\Hat{X}_T} \vert F- G \vert +  A \,  \Vert g - f \Vert_p^{1 \slash n},
 $$
 where $A > 0$ is a constant depending on $X, \theta, n,  p, L$, 
 a uniform bound on  $ \f_0,\psi_0, \dot \f_0, \dot \p_0$ on $X$ 
 and a uniform bound on $\Vert f \Vert_p$ and $\Vert g\Vert_p$. 
 \end{named}

  \medskip

\noindent {\bf D\'edicace.} {\it  C'est un plaisir  de contribuer \`a  ce volume en l'honneur de Jean-Pierre Demailly, 
dont nous appr\'ecions l'exigence et la g\'en\'erosit\'e.
Ses notes de cours \cite{Dem89,Dem11} ont fortement contribu\'e au d\'eveloppement de la th\'eorie du pluripotentiel en France, nous lui en sommes tr\`es reconnaissants.
}

  \medskip
  
\noindent   {\bf Acknowledgement. } We thank the referee for valuable comments which improve the presentation of the paper.

\section{Preliminaries}
 
 \subsection{ Elliptic complex Monge-Amp\`ere equations} Let $(X,\theta)$ be a compact K\"ahler manifold of dimension $n$.  
 
 A function $u: X\rightarrow \mathbb{R}\cup \{-\infty\}$ is called {\it quasi-plurisubharmonic} (quasi-psh for short) on $X$ if 
 it can locally be written $u= \rho +\varphi$, where $\varphi$ is a plurisubharmonic function and $\rho$ is a smooth function. 
 
 A function $u$ is called {\it $\theta$-plurisubharmonic} ($\theta$-psh for short) if it is quasi-psh on $X$ and $\theta+dd^c u \geq 0$ in the weak sense of currents on $X$. The set of all $\theta$-psh functions on $X$ is denoted by ${\rm PSH}(X,\theta)$. 
 
\smallskip 
 
It follows from the seminal work of Bedford and Taylor \cite{Bedford_Taylor_1976Dirichlet,Bedford_Taylor_1982Capacity} 
that 
the Monge-Amp\`ere measure $(\theta+dd^c u)^n=:{\rm MA}_{\theta}(u)$  is well-defined 
when $u$ is a  bounded $\theta$-psh function, and it satisfies several continuity properties.  

We refer the readers to \cite{GZbook2017} for a detailed exposition of global pluripotential theory. In this note we will need the following {\it comparison principle}.

 \begin{pro}
 \label{pro: comparison principle exponential}
 Assume that $u,v \in {\rm PSH}(X,\theta)\cap L^{\infty}(X)$ are such that 
 $e^{- u} {\rm MA}_{\theta}(u) \geq e^{ - v} {\rm MA}_{\theta}(v)$ on $X$. Then $u\leq v$ on $X$. 
 \end{pro}
 
 This result is well known (a proof can be found in \cite[Lemma 2.5]{DDNL16}). 
  We shall also need the following   version of the {\it  domination principle} :

 \begin{pro}\label{pro: domination principle}
  Fix a non-empty open subset $D\subset X$
and let $u,v$ be bounded $\theta$-psh functions on $X$ such that  for all $ \zeta \in \partial D$,
 	\[
 	\limsup_{D\ni z\to \zeta} (u-v)(z)\geq 0. 
 	\]
 	
 	If ${\rm MA}_{\theta}(u)(\{u<v\}\cap D)=0$ then $u\geq v$ in $D$. 
 \end{pro}
 
 \begin{proof}
 	Adding a large constant to both $u$ and $v$, we can assume that $v \geq 0$. For each $\varepsilon>0$ consider $v_{\varepsilon}:= (1-\varepsilon) v$.  Then $v_{\varepsilon} \in {\rm PSH}(X,\theta)$ and $v_{\varepsilon}\leq v$, hence $\limsup_{D\ni z\to \partial D} (u(z)-v_{\varepsilon}(z)) \geq 0$. 
 	The comparison principle (see \cite{CKZ11}) yields
 	\[
 	\int_{\{u<v_{\varepsilon}\}\cap D}\theta_{v_{\varepsilon}}^n \leq \int_{\{u<v_{\varepsilon}\}\cap D} \theta_u^n  \leq \int_{\{u<v\}\cap D} \theta_u^n=0. 
 	\]
 	Since $\theta_{v_\varepsilon}^n \geq \varepsilon^n \theta^n$, we deduce that $u\geq v_{\varepsilon}$ almost everywhere (with respect to Lebesgue measure) in the open set $D$, hence everywhere in $D$. The result follows by letting $\varepsilon\to 0$. 
 \end{proof}

\subsection{Complex Monge-Amp\`ere flows}
 We recall the following definition of (sub/super)solution   which will be used in this note. 
 
 \begin{defi} 
 Let $\f :[0,T[ \times X \longrightarrow \R$ be a function satisfying:
\begin{itemize}
\item $\f$ is continuous in $X_T := [0,T] \times X$,
\item for any $x \in X$, the function $\f (\cdot,x)$ is $\mathcal{C}^1$ in $[0,T]$ and $\dot \f = \partial_t \f$ its partial derivative in $t$ is continuous in $[0,T] \times X$;
\item   for any $t \in [0,T]$ the function $\f_t $ is  bounded  and $\omega_t$-plurisubharmonic.
\end{itemize} 
We say  that the function $\f$ is a
\begin{itemize}
\item solution to the equation \eqref{eq: PMAE} with data $(F,f)$ if for any $t \in ]0,T[$,
\begin{equation*} 
(\omega_t + dd^c \f_t)^n =  e^{\dot {\f} (t,\cdot) + F (t, \cdot, \f_t)} f  d V.
\end{equation*}
\item subsolution to the equation \eqref{eq: PMAE} with data $(F,f)$ if for any $t \in ]0,T[$,
 \begin{equation*}
(\omega_t + dd^c \f_t)^n \geq  e^{\dot {\f} (t,\cdot) + F (t, \cdot, \f_t)} f  d V.
\end{equation*}
\item supersolution to  \eqref{eq: PMAE} with data $(F,f)$ if for any $t \in ]0,T[$, 
\begin{equation*}
(\omega_t + dd^c \f_t)^n \leq   e^{\dot {\f} (t,\cdot) + F (t, \cdot, \f_t)} f  d V.
\end{equation*}
\end{itemize}
 \end{defi} 
 
All the above inequalities have to be understood in the weak sense of currents:
the LHS is a well defined Borel measure since $\f_t$ is a bounded $\omega_t$-psh function
 (see \cite{Bedford_Taylor_1976Dirichlet}),
 while the RHS is a well defined measure which is absolutely continuous with respect to Lebesgue measure.

 \section{Stability in the elliptic case}

 \subsection{A general semi-stability result}
\ref{thm: main thm A} is a consequence of the following more general result which generalizes stability results where the right-hand side does not depend on the unknown function (see \cite{Kol03},  \cite{DZ10}).
 
 \begin{thm}	\label{thm: thm Abis}
     Fix $p>1$. Assume that $0\leq f,g\in L^p(X,dV)$ and $\varphi,\psi$ are  bounded $\theta$-plurisubharmonic  
     functions on $X$  such that 
        \[
        {\rm MA}_{\theta}(\varphi) \geq  e^{\varphi} fdV \
        \text{ and } \   {\rm MA}_{\theta}(\psi) \leq  e^{\psi} gdV. 
        \]
        Then there is a constant $C>0$ depending on $p,n,X,\theta$ and a uniform bound on $\log \Vert f \Vert_p$ and $\log \Vert g\Vert_p$ such that
        \[
        \varphi \leq \psi + \left(C + 2 \mrm{osc}_X \varphi  + 
         2 \mrm{osc}_X \psi\right) \exp{\left(\frac{\mrm{osc}_X \varphi}{n}\right)} \|(g - f)_+\|_p^{1/n}.  
        \]
        
\end{thm}

\begin{proof}
We use a perturbation argument inspired by an idea of Ko{\l}odziej \cite{Kol96} (see also  \cite[proof of Theorem 3.11]{cuong2014weak})  who considered the local case. 

  For simplicity we normalize $\theta$ and $dV$ so that $\int_X dV = {\rm Vol}(\theta) = 1$, 
  and we denote by $\|f\|_p$ the $L^p$-norm of $f$ with respect to the volume form $dV$. 
  We assume that $\|f\|_p,\|g\|_p$ are uniformly bounded away from zero and infinity
  (i.e. $\log \|f\|_p, \log \|g\|_p$ are uniformly bounded).
    
If $\Vert (g-f)_+\Vert_p = 0$ then $g \leq f$ almost everywhere in $X$. In this case, $\varphi$ is a subsolution and $\psi$ is a supersolution to the same complex Monge-Amp\`ere equation. Then the comparison principle  (Proposition \ref{pro: comparison principle exponential}) yields $\f \leq \psi$ in $X$, which proves the result. 

We assume in the sequel that $\|(g-f)_+\|_p>0$.
Integrating the inequality ${\rm MA}_\theta (\f) \geq e^\f f d V$, we see that 
\[
\inf_X \varphi  \leq - \log \left (\int_X f dV \right ) \leq -\log \|f\|_p, 
\]
hence $\sup_X \f \leq {\rm osc}_X  \f - \log \|f\|_p$.
Similarly   $\sup_X \psi \geq - \log \|g\|_p$, hence $- \inf_X \psi \leq {\rm osc}_X \psi + \log \|g\|_p$.

 Set $\varepsilon:= e^{\sup_X \varphi/n} \|(g - f)_+\|_p^{1/n}$. If $\varepsilon \geq 1\slash 2$ then 
\begin{flalign*}
	\sup_X &(\varphi-\psi) \leq \sup_X \varphi - \inf_X \psi \\
	& \leq  2 \left( \sup_X \varphi - \inf_X \psi \right) \exp{\left(\frac{\sup_X \varphi}{n}\right)} \|(g - f)_+\|_p^{1/n},
\end{flalign*}
as desired. 

So we can assume that $\varepsilon<1\slash 2$.   H\"older inequality yields
\[
\int_X \frac{(g - f)_+}{\|(g - f)_+\|_p} dV \leq 1. 
\]  

It follows from  \cite{Kolodziej_1998Monge} that there exists a bounded  $\theta$-psh function $\rho$  such that 
\begin{equation*}  
{\rm MA}_{\theta}(\rho) = hdV := \left (a + \frac{(g - f)_+}{\|(g - f)_+\|_p}\right)dV, \ \sup_X \rho =0,
\end{equation*}
where $a\geq 0$ is a normalization constant to insure that $\int_X hdV = \int_X dV$.  
Since $\|h\|_p\leq 2$ the uniform estimate of Kolodziej \cite{Kolodziej_1998Monge}  guarantees 
\[
-C_1\leq \rho \leq 0,
\] where $C_1$ only  depends on $p,\theta,n$.   

Observe that one can use here soft techniques and avoid the use of Yau's 
Theorem (see \cite{eyssidieux2011viscosity,  GZbook2017}). Consider now 
\[
\varphi_{\varepsilon}:= (1-\varepsilon)\varphi + \varepsilon \rho -C_2 \varepsilon + n\log (1-\varepsilon),
\]
 where $C_2$ is a positive constant to be specified hereafter, and $\varepsilon<1/2$ is defined as above. 
 Then $\varphi_{\varepsilon}$ is a  bounded $\theta$-psh function, and a direct computation shows that 
\begin{eqnarray*}
        {\rm MA}_{\theta}( \varphi_{\varepsilon}) & \geq & (1-\varepsilon)^n{\rm MA}_{\theta}(\varphi) + \varepsilon^n {\rm MA}_{\theta}(\rho)\\
       &  \geq & e^{\varphi + n\log (1-\varepsilon)} fdV+ e^{\varphi}(g - f)_+ dV.  
\end{eqnarray*}
We choose $C_2=-\inf_X \varphi$ so  $\rho-\varphi\leq C_2$ and 
$\varphi_{\varepsilon} \leq \varphi + n\log (1-\varepsilon) \leq \varphi$. Thus
\begin{equation*}  
{\rm MA}_{\theta}( \varphi_{\varepsilon}) \geq e^{\varphi_{\varepsilon}} (f+ (g- f)_+)dV \geq e^{\varphi_{\varepsilon}} g dV.        
\end{equation*} 
In other words, $\varphi_{\varepsilon}$ is a subsolution and $\psi$ is a supersolution to the equation ${\rm MA}_{\theta}(\phi)=e^{\phi}gdV$. The comparison principle (Proposition \ref{pro: comparison principle exponential})
insures that $\varphi_{\varepsilon}\leq \psi$, hence
\begin{eqnarray*}
        \varphi-\psi &= &\varphi_{\varepsilon}-\psi + \varepsilon (\varphi-\rho)  + C_2\varepsilon - n\log (1-\varepsilon)\\
        &\leq & (C_1 -\inf_X \f+ \sup_X \varphi - \inf_X \psi + 2n) \varepsilon\\
        &= & (C_1+{\rm osc}_X  \varphi - \inf_X \psi +2n) \exp{\left (\frac{\sup_X \varphi}{n}\right )} \|(g-f)_+\|_p^{1/n}.  
\end{eqnarray*}
Since $C_1$ only depends on $p,\theta,n$,  the result follows. 
\end{proof}

\subsection{The proof of Theorem A}
Without loss of generality we normalize $f,g$ and $\theta$ so that $\int_X fdV =\int_X gdV=\int_X \theta^n= \int_X dV$. Let $\phi$ be the unique bounded  $\theta$-plurisubharmonic function on $X$ normalized by $\sup_X \phi=0$ such that ${\rm MA}_{\theta}(\phi) =fdV$. The existence of $\phi$ follows from Ko{\l}odziej's celebrated work \cite{Kolodziej_1998Monge} and moreover, $-C\leq \phi \leq 0$, where $C$ is a uniform constant depending on $X,\theta, p, \|f\|_p$.  Then $\phi$ is a subsolution while $\phi+C$ is a supersolution to the Monge-Amp\`ere equation 
\[
{\rm MA}_{\theta}(u) = e^u fdV. 
\]
It thus follows from the comparison principle (Proposition \ref{pro: comparison principle exponential}) that 
\[
\phi\leq \varphi \leq \phi+C.
\] 
We thus obtain a uniform bound for $\varphi$. The same arguments give a uniform bound for $\psi$ as well.  
\ref{thm: main thm A} follows therefore from Theorem \ref{thm: thm Abis}.

\begin{remark} \label{rem:improve} In \ref{thm: main thm A}, one can replace the 
$L^p$-norm of $f - g$ by the $L^1$-norm, at the cost of decreasing the exponent $1 \slash n$ to $1 \slash (n + \e)$ ($\e >0$ arbitrarily small). 

Namely under the  assumptions of \ref{thm: main thm A}, for any $\e> 0$, there exists a constant $C > 0$ which depends on $p,n,X,\theta, \e$ and a uniform bound on $\Vert f \Vert_p$ and $\Vert g\Vert_p$ such that
\begin{equation*}
   \Vert \varphi- \psi \Vert_{\infty} \leq   C \|f-g\|_1^{1/(n + \e)}.  
\end{equation*}

Indeed repeating the proof with an exponent $1< r < p$ close to $1$,  we get a bound in terms of $\Vert f - g\Vert_r^{1 \slash n}$.
Then observe that if we write $ r = (1-t) + t p$, by H\"older-interpolation inequality applied to  $h := \vert f - g\vert \in L^p (X)$, we have
$$
 \int_X h^r d V \leq \left(\int_X h d V \right)^{1 - t} \left(\int_X h^p \right)^{t},
$$
which implies that 
$$
\Vert h \Vert_r \leq \Vert h\Vert_1^{(1 - t)\slash r} \Vert h\Vert_p^{t p \slash r}.
$$
Since  $t = (r - 1) \slash (p - 1)$ and $1 - t = (p - r) \slash (p - 1)$, we see that the exponent $(1 - t)\slash r = (p - r) \slash r (p - 1)$ is arbitrary close to $1$ as $r \to 1$.

 \end{remark}

\section{Stability in the parabolic case}

\subsection{A parabolic comparison principle}

We establish in this section a maximum principle which is classical when the data are smooth.
It has been obtained for continuous data in \cite{EGZ16}, we propose here a different approach which
applies to our present setting:
 
 \begin{theo} \label{thm:CP}
 Fix $\f : [0,T[ \times X \longrightarrow \R$ a subsolution, $\psi :[0,T[ \times X \longrightarrow \R$ a  supersolution to the parabolic equation \eqref{eq: PMAE}, where $0 \leq f \in L^p (X)$ with $p > 1$.  
 Then 
$$
 \sup_{X_T} (\f - \psi) \leq \sup_{X} (\f_0 - \psi_0)_+.
$$
 \end{theo}
 
 Recall that  $(\f_0 - \psi_0)_+ := \sup \{\f_0 - \psi_0 , 0\}$.

 \begin{proof} 
 Fix $T' < T$,  $0 < \varepsilon$.
 We first assume that $ M_0 := \sup_{X} (\f_0 - \psi_0) \leq 0 $  and we are going   prove that
  $\f \leq \psi+2\varepsilon t $ in $X_{T'}$.  
	
	Consider $w (t,x):=\f(t,x)-\psi(t,x)-2\varepsilon t$.
	This function is upper semi-continuous on the compact space $[0,T'] \times X$. Hence $w$ attains a maximum at some point $(t_0,x_0) \in [0,T'] \times X$. We claim that $w (t_0,x_0)\leq 0$. Assume by contradiction that $ w (t_0,x_0) >  0$,
	in particular  $t_0>0$, and set
	 $$
	 K := \{ x \in X ; w (t_0,x) = w (t_0,x_0)\}.
	 $$
 
The classical maximum principle insures	 that for all $x \in K$,
	\[
	\partial_t \f(t_0,x) \geq \partial_t \psi(t_0,x) +2\varepsilon.
	\]
By continuity of the partial derivatives in $(t,x)$, we can find an open neighborhood $D$ of $K$ such that for all $x\in D$
	\[
	\partial_t \f (t_0,x) > \partial_t \psi(t_0,x) +\varepsilon.
	\]
Set $u := \f (t_0,\cdot)$ and $v := \psi (t_0,\cdot)$. Since $\f$ is a subsolution  and  $\psi$ is a supersolution to (\ref{eq: PMAE})  we infer
	\[
	(\omega_{t_0} + dd^c u)^n \geq e^{F(t_0,x,u (x))-F(t_0,x,v (x))+\varepsilon}(\omega_{t_0} + dd^c v)^n,
	\] 
	in the weak sense of measures in $D$.  Recall that
\begin{itemize}
\item $u$ and $v$ are continuous on $D$,
\item   $F$ is non-decreasing in $r$,
\item $u (x) > v(x)+\varepsilon t_0$ for any $x \in K$. 
\end{itemize}
Shrinking $D$ if necessary,  we can assume that the latter inequality is true in $D$.
We thus get
	\[
	(\omega_{t_0} + dd^c u)^n \geq  e^{\varepsilon}(\omega_{t_0} + dd^c v)^n.
	\]
From this we see in particular that $D\neq X$, hence $\partial D\neq \emptyset$. 	
	
Consider now $\tilde{u}:=u +\min_{\partial D}(v -u)$. 
Since $v \geq \tilde{u}$ on $\partial D$, Proposition \ref{pro: domination principle} yields
	\[
	\int_{\{v<\tilde{u}\}\cap D}e^{\varepsilon} (\omega_{t_0} + dd^c v)^n \leq \int_{\{v<\tilde{u}\}\cap D} (\omega_{t_0} + dd^c u)^n \leq \int_{\{v<\tilde{u}\}\cap D} (\omega_{t_0} +dd^c v)^n. 
	\]
	
It then follows that $\tilde{u} \leq v$, almost everywhere in $D$ with respect to  the measure  $(\omega_{t_0} + dd^c v)^n$, hence everywhere in $D$ by the domination principle (see Proposition \ref{pro: domination principle}).  
In particular for all $x \in D$,
	\begin{equation}
		\label{eq: varphi r less than psi}
		u (x)- v(x) +\min_{\partial D}(v - u)=\tilde{u}(x) -v (x) \leq 0, 
	\end{equation}
Since $K\cap \partial D =\emptyset$, we infer $w(t_0,x) < w(t_0,x_0)$, for all $x\in \partial D$, i.e.
	$$
	 u (x) - v (x) < u (x_0) - v (x_0)
	 \text{ for all }
	 x\in \partial D ,
	 $$
	   contradicting \eqref{eq: varphi r less than psi}. 
Altogether this shows that $t_0 = 0$, thus $\varphi \leq \psi+2\varepsilon t$ in $X_{T'}$. Letting $\varepsilon\to 0$ and $T' \to T$ we obtain that $\varphi \leq \psi$ in $X_{T}$.
	
We finally get rid of the assumption $\f_0 \leq \p_0$.
If $M_0 := \sup_X (\f_0 - \psi_0) > 0$ then $\f - M_0$ is a subsolution of the same equation since $F$ is non decreasing in the last variable. Hence  $\f - M_0 \leq \psi$ in $X_{T}$.
This proves the required inequality.
 \end{proof}

 \begin{remark}
 We note for later works that the above proof only requires $t \mapsto \f(t,\cdot),\p(t,\cdot)$ to be ${\mathcal C}^1$
 in $]0,T[ \times X$. This should be useful in analyzing the smoothing properties of complex Monge-Amp\`ere flows
 at time zero. 
 \end{remark}

 \subsection{A parabolic semi-stability theorem}
 
We now establish a technical comparison principle which is a key step in the proof of \ref{thm: main thm B}. The proof of this result does not require the smoothness assumption on $F,G,f,g$. 
We assume in this subsection that 
 \begin{itemize}
\item  $ F, G : \Hat{ X}_T := [0,T[ \times X \times \R \rightarrow \R $ are continuous;
\item $F,G$ are non decreasing in the last variable;
 \item $F,G$ satisfy   condition (\ref{eq:Lip}) with the same constant $L > 0$.
 \item $0 \leq f, g \in L^p (X)$  with $p > 1$.
 \end{itemize}

  \begin{thm} 
  Assume that $\f : [0,T[ \times X \longrightarrow \R$ is a subsolution to the parabolic equation \eqref{eq: PMAE} with data $(F,f)$ and $\psi :[0,T[ \times X \longrightarrow \R$ is a supersolution to the parabolic equation \eqref{eq: PMAE} with data $(G,g)$.
 Then 
 \begin{equation*}
 \sup_{X_T} (\f - \psi) \leq \sup_{X} (\f_0 - \psi_0)_+ + T \, \sup_{\Hat{X}_T} (G-F)_+  + A \, \Vert (g - f)_+ \Vert_p^{1 \slash n},
 \end{equation*}
 where $A>0$ depends on $X, \theta, n,  p$ and a uniform bound on $\dot \f $,  $ \f ,\psi$, and $\sup_{X_T} G(t,x, \sup_{X_T} \varphi)$. 
 \end{thm}
 
\begin{remark}
	In the second term of the estimate above  one can replace $\sup_{\Hat{X}_T} (G-F)_+$ by 
	\[
	\sup_{[0,T[ \times X \times I} (G-F)_+,
	\]
	where $I= [\inf_{X_T} \varphi, \sup_{X_T} \varphi]$ is a compact interval in $\mathbb{R}$.
\end{remark}
 
 \begin{proof}  
 We first assume that $\Vert (g - f)_+ \Vert_p > 0$.
  Since  $\int_X d V = \int_X \theta^n = 1$, it follows from \cite{Kolodziej_1998Monge}
that  there exists  $\rho \in {\rm PSH} (X,\theta) \cap \mathcal{C}^0(X)$ such that
 \begin{equation} \label{eq:yaubis}
 (\theta + dd^c \rho)^n = \left(a +  \frac{(g - f)_+}{\Vert (g - f)_+ \Vert_p}\right) d V
 \end{equation}
normalized by  $\max_X \rho = 0$, where $a \geq 0$ is a normalizing constant given by
 $$
 a := 1 - \frac{\Vert (g - f)_+\Vert_1}{\Vert (g - f)_+ \Vert_p} \in [0,1].
 $$
 
 We moreover have a uniform bound on $\rho$  which only depends on the $L^p$ norm of the density
 of $(\theta + dd^c \rho)^n$ which is here bounded from above by $2$,
 \begin{equation} \label{eq:est0}
 \Vert \rho \Vert_{\infty} \leq C_0 (a + 1) \leq 2 C_0,
 \end{equation}
 where $C_0 > 0$ is a uniform constant depending only on $(X,\theta, p)$.
 
 Fix  $B, M > 0$.
For $0 < \delta < 1 $ and $(t,x) \in X_T$ we set
 $$
 \f_\delta (t,x) := (1 - \delta) \f (t,x) + \delta  \rho + n \log (1 - \delta) - B \delta t - M t.
 $$
 The plan is to choose $B,M>0$ in such a way that $\f_{\delta}$ be a subsolution for 
 the parabolic equation \eqref{eq: PMAE} with data $(G,g)$. The conclusion will then follow
 from the comparison principle (Theorem \ref{thm:CP}).

 Observe that for $t \in ]0,T[$ fixed, $\f_\delta(t,\cdot) $ is $\omega_t$-plurisubharmonic in $X$ and 
 $$
 (\omega_t + dd^c \f_{\delta} (t,\cdot))^n \geq (1 - \delta)^n (\omega_t + dd^c \f_t)^n + \delta^n (\theta + dd^c \rho)^n.
 $$
Using that $\f$ is a subsolution to  (\ref{eq: PMAE}) with density $f$, we infer
 \begin{equation} \label{eq:subest}
 (\omega_t + dd^c \f_{\delta} (t,\cdot))^n \geq   e^{\dot \f + F (t,\cdot,\f) + n \log (1- \delta)} f d V + \delta^n \frac{(g - f)_+}{\Vert (g - f)_+ \Vert_p} d V.
 \end{equation}

Set $m_0 := \inf_{X_T} \f$, $m_1 := \inf_{X_T} \dot{\f}$ and choose $M:= \sup_{\Hat{X}_T} (G-F)_+$.
Noting that $\f \geq \f_{\delta} + \delta \f$ and recalling that $G$ is non decreasing in the last variable, we obtain
 \begin{eqnarray*}
&& \dot{\f} (t,x)  + F (t,x ,\f (t,x)) + n \log (1-\delta) \\
&\geq & \dot{\f}_\delta (t,x) + \delta \dot{\f} (t,x) + G \left(t,x, \f (t,x)\right) -M +
  n \log (1-\delta) + B \delta +M \\
 &\geq & \dot{\f}_\delta (t,x) + \delta \dot{\f} (t,x) + G \left(t,x,\f_\delta (t,x) + \delta \f (t,x)\right) +
  n \log (1-\delta) + B \delta \\
 & \geq & \dot{\f}_\delta (t,x) + \delta m_1 + G \left(t,x,\f_\delta (t,x) + \delta m_0 \right) + n \log (1-\delta) + B \delta. 
 \end{eqnarray*}
The Lipschitz condition (\ref{eq:Lip}) yields
 \begin{eqnarray*}
&&\dot{\f} (t,x)  + F (t,x ,\f (t,x)) + n \log (1-\delta) \\
&\geq & \dot{\f}_\delta  (t,x) + G (t,x,\f_\delta (t,x))
+ B \delta - L \delta m_0 + \delta m_1 + n \log (1-\delta).
 \end{eqnarray*}
Using the elementary inequality $\log (1 - \delta) \geq - 2 (\log 2 ) \delta$ for  $0 < \delta \leq 1 \slash 2$, 
it follows that for $0 < \delta \leq 1 \slash 2$,
 $$
 B \delta - L \delta m_0  + \delta m_1 + n \log (1-\delta) \geq  (B  - L  m_0 +  m_1 - 2 n \log 2) \delta.
 $$
  
\smallskip  
  
We now choose $B :=  L  m_0  - m_1 +  2 n \log 2$ so that
  $$
  \dot{\f} (t,x)  + F (t,x ,\f (t,x)) + n \log (1-\delta) \geq  \dot{\f}_\delta  (t,x) + G (t,x,\f_\delta (t,x)),
  $$
   which, together with  (\ref{eq:subest}),  yields
\begin{equation}  \label{eq: subest2} 
 (\omega_t + dd^c \f_{\delta} (t,\cdot))^n \geq   e^{\dot{\f}_\delta  (t,\cdot) + G (t,\cdot,\f_\delta (t,\cdot))} f d V + \delta^n \frac{(g - f)_+}{\Vert (g - f)_+ \Vert_p}.
 \end{equation} 
On the other hand,  if we set 
$$
M_1 := \sup_{X_T}  \dot\f , \;
M_0 := \sup_{X_T} \f
\text{ and }
N := \sup_{X_T} G (t,x,M_0),
$$
then the properties of $G$ insure
\begin{eqnarray*}
\dot{\f}_\delta (t,x) + G (t,x,\f_\delta (t,x)) & \leq & (1- \delta) \sup_{X_T} \dot{\f} + \sup_{X_T} G (t,x, (1-\delta) \f (t,x))  \\
& \leq & (1- \delta) M_1 + \sup_{X_T} G (t,x, (1-\delta)M_0) \\
& \leq & (1- \delta) M_1 +  M_0 L \delta + N \leq N + \max \{L M_0, M_1\}
\end{eqnarray*}
Using  \eqref{eq: subest2} we conclude that for $0 < \delta < 1 \slash 2$,
\begin{equation} \label{eq:subest3}
(\omega_t + dd^c \f_{\delta} (t,\cdot))^n \geq   e^{\dot{\f}_\delta  (t,\cdot) + G (t,\cdot,\f_\delta (t,\cdot))} \left(f  + \delta^n e^{ - M_2}\frac{(g- f)_+}{\Vert (g - f)_+ \Vert_p}\right) d V,
\end{equation}
 where $M_2 :=  N + \max \{L M_0, M_1\}$.

To conclude that $\f_{\delta}$ is a subsolution, we finally set
 \begin{equation} \label{eq:est1}
 \delta  := \Vert (g - f)_+\Vert^{1 \slash n}_p e^{ M_2 \slash n}\cdot
 \end{equation}
 Assume first that $\Vert (g - f)_+\Vert_p \leq 2^{- n} e^{ - M_2}$  so that  $\delta \leq 1 \slash 2$. 
 It follows from (\ref{eq:subest3}) that
 \begin{eqnarray*}
 (\omega_t + dd^c \f_{\delta} (t,\cdot))^n & \geq &   e^{\dot{\f}_\delta  (t,\cdot) + G (t,x,\f_\delta (t,\cdot))} (f  +  (g - f)_+) d V \\
 & \geq & e^{\dot{\f}_\delta  (t,\cdot) + G (t,x,\f_\delta (t,\cdot))} g d V,
 \end{eqnarray*}
 hence  $\f_\delta$ is a subsolution to   \eqref{eq: PMAE} for the data   $(G,g)$ in $D$.
 The comparison principle (Theorem \ref{thm:CP}) insures that for all $(t,x) \in X_T$,
 $$
 \f_{\delta} (t,x) -\psi (t,x) \leq \max_X (\f_\delta (0,\cdot) - \psi (0,\cdot)_+).
 $$
Taking into account the estimates (\ref{eq:est0}) and (\ref{eq:est1}),  we get
\begin{equation*} 
\f (t,x)- \psi (t,x) \leq \max_X (\f (0,\cdot) - \psi (0,\cdot)_+ + T M +  A_1
 \Vert (g - f)_+\Vert^{1 \slash n}_p, 
\end{equation*}
where 
$$ 
A_1 :=  (M_0 +   2 C_0 + 2 n \log 2 + B  T) e^{M_2 \slash n}\cdot
$$ 
 
 When $\Vert (g - f)_+\Vert_p >  2^{- n} e^{- M_2}$, we can choose a constant $A_2 > 0$ so that  
 $$
 \f (t,x)- \psi (t,x) \leq \max_X (\f _0 - \psi_0)_+ + A_2 2^{- n} e^{- M_2}.
 $$ 
We eventually take $A = \max \{A_1,A_2\}$.

\smallskip
 
Assume finally that $\Vert (g-f)_+\Vert_p = 0$ which means that $g \leq f$ almost everywhere in $X$.
 In this case we solve the
 equation (\ref{eq:yaubis}) with the right hand side equal to $dV$ and repeat the same arguments with an arbitrary $\delta > 0$. 
The conclusion  follows by letting $\delta \to 0$.
 \end{proof}

 \subsection{Proof of Theorem B}

 The proof of Theorem B goes by symmetrizing the roles of $\f$ and $\p$, and establishing uniform bounds on
 $\f,\p,\dot{\f},\dot{\p}$  depending on uniform bounds 
 for $\f_0,\dot{\f_0},\p_0,\dot{\p_0}$ and  $||f||_p,||g||_p$. 
 
 \subsubsection{Bounds on $\f$}
 By assumption, we can fix $a, A>0$ such that 
 $$a \theta \leq \omega_t \leq A \theta , \ \forall (t,x) \in X_T.$$
 
Let $\rho$  be  the unique bounded normalized $\theta$-psh function $\sup_X \rho =0$ such that ${\rm MA}_{\theta}(\rho)=C_0fdV$, where $C_0>0$ is a uniform normalization constant.  Then $\rho$ is bounded by a constant depending only on the $L^p$ norm of $f$ and on $\theta, p$.   
	We set 
	\[
	C_1 := \sup_X (a\rho -\varphi_0) \ {\rm and} \  C_2:= \sup_X (\varphi_0-A\rho). 
	\]
	We next introduce the following uniform constants
	\[
	C_3 := \sup_{(t,x)\in [0,T] \times X}  F(t,x, \varphi_0(t,x))  \ {\rm and}\  C_4 := \inf_{(t,x) \in [0,T] \times X}  F(t,x, \varphi_0(t,x)). 
	\]
	A direct computation shows that the function defined on $X_T$ by
	\[
	u(t,x) := a\rho - C_1 -\max(C_3 - n \log a -\log  C_0,0)t
	\]
	is a subsolution to the parabolic equation \eqref{eq: PMAE} with data $(f,F)$. Indeed, for fixed $t \in ]0,T[$, the Monge-Amp\`ere measure of $u_t$ can be estimated as 
	\[
	(\omega_t + dd^c u_t)^n \geq a^n (\theta +dd^c \rho)^n = a^n C_0 fdV. 
	\]
	Since $F$ is non-decreasing in the last variable and $\varphi_0 \geq u_t$ 
	we obtain $C_3 \geq F(t,x,\varphi_0) \geq F(t,x,u_t)$. Thus
	\begin{eqnarray*}
		e^{\dot{u}_t + F(t,\cdot,u_t)} fdV &\leq & e^{-\max(C_3 - n \log a -\log  C_0,0) + C_3} fdV  \\
		&\leq & a^n C_0 fdV \leq (\omega_t+dd^c u_t)^n.
	\end{eqnarray*}
	
	 A similar computation shows that  the function defined on $X_T$ by
	\[
	v (t,x) := A \rho + C_2  + \max(-C_4 + n \log A + \log C_0,0)t
	\]
	is a supersolution to the parabolic equation 
	\begin{equation}
		\label{eq: PCMA A theta}
		(A\theta +dd^c v_t)^n = e^{\dot{v}_t +F(t,x,v_t)} fdV.
	\end{equation}
	Since $A\theta \geq \omega_t$ we see that $\varphi_t$ is a subsolution to the equation \eqref{eq: PCMA A theta}. 
The parabolic comparison principle (Theorem \ref{thm:CP}) therefore yields  
	\[
	u \leq \varphi \leq v.,
	\] 
	in $X_T$.
 
 \subsubsection{Bounds on $\dot{\f}$}
 
 We now provide a uniform bound on $\dot{\f}_t$, assuming all data are smooth.
 We only outline the proof since the arguments are classical.  
 The previous subsection has provided a uniform bound
 \[
 -B_0 \leq \varphi(t,x) \leq B_0,  \ \forall (t,x) \in [0,T[ \times X. 
 \]
 Since $\omega_t$ is smooth in $t$ and $\omega_t$ is uniformly K\"ahler we can fix a positive constant $B_1$ such that 
 \[
 -B_1 \omega_t \leq \dot{\omega}_t \leq B_1 \omega_t. 
 \] 
  Up to enlarging $B_1$ we can further assume that 
 \[
 -B_1 \leq \partial_r F(t,x,r) \leq B_1, \ \forall (t,x,r) \in [0,T[ \times X \times [-B_0,B_0]. 
 \]
 Set
 $$
   \Delta_t(h)=\Delta_{\omega_t +dd^c \f_t} (h)=
   n \frac{dd^c h \wedge (\omega_t +dd^c \f_t)^{n-1} }{ (\omega_t +dd^c \f_t)^{n}},
 $$
 and 
 $$
 {\rm Tr}_t(\eta):=n \frac{\eta \wedge (\omega_t +dd^c \f_t)^{n-1} }{ (\omega_t +dd^c \f_t)^{n}}\cdot
 $$
 A straightforward computation yields
 \begin{eqnarray*}
  \ddot{\f_t} &= &
  \Delta_t (\dot{\f_t})+ {\rm Tr}_t(\dot{\omega_t})-\partial_t F(t,x,\f_t)  -\dot{\f_t} \partial_rF(t,x,\f_t) \\
  &\leq & 
  \Delta_t (\dot{\f_t})-B_1 \Delta_t(\f_t) +C -\dot{\f_t} \partial_rF(t,x,\f_t),
 \end{eqnarray*}
using that $\f_t$, hence $\partial_t F(t,x,\f_t)$, is uniformly bounded on $X_T$, and that
$$
{\rm Tr}_t(\dot{\omega_t}) \leq B_1 {\rm Tr}_t(\omega_t)=B_1n-B_1 \Delta_t(\f_t).
$$

 Consider
 $$
 H(t,x):=\dot{\f_t}(x)-B_1 \f_t(x)-(C+1)t.
 $$
It follows from the above computation that
$$
\left( \frac{\partial}{\partial t}-\Delta_t \right) H \leq -1-[B_1+\partial_rF(t,x,\f_t)] \dot{\f_t} .
$$
If $H$ reaches its maximum at time zero, then 
 $$
 H \leq \sup_X \dot{\f_0}-B_1 \inf_X \f_0,
 $$
 hence $\dot{\f_t} \leq C(T,B_0,B_1)+\sup_X  \dot{\f_0}$. If $H$ reaches its maximum at $(t_0,x_0)$ with $t_0>0$, we obtain at $(t_0,x_0)$,
 $$
 0 \leq \left( \frac{\partial}{\partial t}-\Delta_t \right) H \leq -1-[B_1+\partial_rF(t,x,\f_t)] \dot{\f_t},
 $$
 hence $\dot{\f_t}(t_0,x_0) \leq 0$ since $[B_1+\partial_rF(t,x,\f_t)] \geq 0$. Therefore
 $$
 H_{max} =H(t_0,x_0) \leq -B_1 \inf_X \f_{t_0} \leq B_0 B_1,
 $$
 and we obtain an appropriate bound from above for $\dot{\f_t}$.
 
 \smallskip
 
 The proof for the lower bound goes along similar lines, considering 
 $G=\dot{\f_t}(x)+B_1 \f_t(x)+(C'+1)t$.

\section{Concluding remarks}

 \subsection{Varying the reference forms}

It is certainly interesting to  study the stability properties when the reference forms $\theta,\omega_t$ are varying.
For simplicty we address this issue here only in the elliptic case.

\begin{theorem} 
Fix $\theta,\omega$ K\"ahler forms.
Fix $p>1$ and $0\leq f,g\in L^p(X,dV)$. If  $\varphi$ (resp. $\p$) is a  bounded $\theta$-plurisubharmonic  
(resp. $\omega$-plurisubharmonic) function on $X$  such that 
        \[
        {\rm MA}_{\theta}(\varphi) =  e^{\varphi} fdV 
        \text{ and }
        \   {\rm MA}_{\omega}(\psi) = e^{\psi} gdV,
        \]
        then
        \[
        \Vert \varphi- \psi \Vert_{\infty} \leq   C \left\{ \|f-g\|_p^{1/n}+d(\omega,\theta) \right\},
        \]
    where $C>0$  depends on $p,n,X,\omega,dV$ and uniform bounds on   $\Vert f \Vert_p, \Vert g \Vert_p$.
\end{theorem} 

We use here the following distance on positive forms,
    \[
    d(\omega,\theta):= \inf \{ t>0 \  ; \  e^{-t}\omega \leq \theta \leq e^t \omega\}. 
    \]

\begin{proof}
Set
$$
c=\inf \{ t>0 \  ; \  (1-t) \omega \leq \theta \leq (1+t)\omega\}. 
$$
Adjusting the constant we can assume that $c\leq 1/2$ and $c \simeq d(\omega,\theta)$. Now 
\[
\psi_c := (1-c) \psi + n \log (1-c) + c\inf_X \psi
\]
is a $\theta$-psh function whose Monge-Amp\`ere measure can be estimated as 
\[
{\rm MA}_{\theta}(\psi_{c}) \geq (1-c)^n e^{\psi} gdV\geq e^{\psi_c} gdV. 
\]
It thus follows from Theorem \ref{thm: thm Abis} that 
\[
\psi_{c} \leq \varphi + C \|f-g\|_p^{1/n},
\]
for a uniform constant $C$. Note that the uniform norm of $\psi_c$ is uniformly controlled by $\|\psi\|$ because $c\leq 1/2$.  From this and the definition of $\psi_c$ we obtain 
\[
\psi \leq \varphi + C' (\|f-g\|_p^n + c),
\]
where $C'$ is a uniform constant. Exchanging the roles of $\varphi$ and $\psi$ yields the conclusion. 
\end{proof}

\subsection{Big classes}

The ideas we have developed so far can also be applied to the more general setting of   cohomology classes 
that are merely big  rather than K\"ahler.
We briefly explain the set up for the elliptic stability.

\smallskip

We assume that $\theta$ is a smooth form representing a big cohomology class.
A $\theta$-psh function $\f$ is no longer bounded on $X$, but it can have {\it minimal singularities}.
The function
$$
V_{\theta}=\sup \{ u \in {\rm PSH}(X,\theta) \; ; \; u \leq 0 \}
$$
is an example of $\theta$-psh with minimal singularities. Any other $\theta$-psh function $\f$ with minimal singularities
satisfies $||\f-V_{\theta}||_{L^{\infty}(X)}<+\infty$.

The pluripotential theory in big cohomology classes has been developed in \cite{Boucksom_Eyssidieux_Guedj_Zeriahi_2010_Big}.
In short, the Bedford-Taylor theory can be developed, replacing $X$ by the {\it ample locus} of $\theta$,
a Zariski open subset in which $V_{\theta}$ is locally bounded.

The a priori estimate of Ko{\l}odziej can be extended, as well as \ref{thm: main thm A}.
It suffices indeed to establish the following:

\begin{thm} 
                Fix $p>1$ and assume that $0\leq f,g\in L^p(X,dV)$, and $\varphi,\psi$ are  
                $\theta$-psh  functions on $X$ with minimal singularities such that 
        \[
        {\rm MA}_{\theta}(\varphi) \geq  e^{\varphi} fdV \ ; \   {\rm MA}_{\theta}(\psi) \leq  e^{\psi} gdV. 
        \]
        Then there is a constant $C>0$ depending on $p,n,X,\theta$ such that
        \[
        \varphi \leq \psi + \left (C + 2\sup_X |V_{\theta}-\varphi|  + 2 \sup_X \max(V_{\theta}-\psi,0) \right) \exp{\left(\frac{\sup_X \varphi}{n}\right)} \|f-g\|_p^{1/n}.  
        \]
\end{thm}

\begin{proof}
Set $\varepsilon:= e^{\sup_X \varphi/n} \|f-g\|_p^{1/n}$. If $\varepsilon \geq 1/2$ then 
\begin{flalign*}
	\sup_X &(\varphi-\psi) \leq  \left  ( \sup_X| \varphi-V_{\theta}| + \sup_X \max(V_{\theta}-\psi,0) \right ) \\
	& \leq   2\left( \sup_X| \varphi-V_{\theta}| + \sup_X \max(V_{\theta}-\psi,0) \right) \exp{\left(\frac{\sup_X \varphi}{n}\right)} \|f-g\|_p^{1/n},
\end{flalign*}
as desired. So we can assume that $\varepsilon<1/2$.   H\"older inequality yields
\[
\int_X \frac{|f-g|}{\|f-g\|_p} dV \leq 1. 
\]  

Let $\rho$  be the unique $\theta$-psh function with minimal singularities such that 
\[
{\rm MA}_{\theta}(\rho) = hdV := \left (a + \frac{|f-g|}{\|f-g\|_p}\right)dV, \ \sup_X \rho =0,
\]
where $a\geq 0$ is a normalization constant to insure that $\int_X hdV = \int_X dV$.  

Since $\|h\|_p\leq 2$, it follows from \cite[Theorem 4.1]{Boucksom_Eyssidieux_Guedj_Zeriahi_2010_Big}  that 
\[
-C_1\leq \rho-V_{\theta} \leq 0,
\] where $C_1$ only  depends on $p,\theta,n$.  Consider now 
\[
\varphi_{\varepsilon}:= (1-\varepsilon)\varphi + \varepsilon \rho -C_2 \varepsilon + n\log (1-\varepsilon),
\]
 where $C_2$ is a positive constant to be specified hereafter.  Then $\varphi_{\varepsilon}$ is a  $\theta$-psh function with minimal singularities, and a direct computation shows that 
\begin{eqnarray*}
        {\rm MA}_{\theta}( \varphi_{\varepsilon}) & \geq & (1-\varepsilon)^n{\rm MA}_{\theta}(\varphi) + \varepsilon^n {\rm MA}_{\theta}(\rho)\\
       &  \geq & e^{\varphi + n\log (1-\varepsilon)} fdV+ e^{\varphi}|f-g| dV.  
\end{eqnarray*}
If we choose $C_2=\sup_X (V_{\theta}-\varphi)$ then $\rho-\varphi\leq C_2$ and $\varphi_{\varepsilon} \leq \varphi + n\log (1-\varepsilon) \leq \varphi$. So we can continue the above estimate to arrive at 
\begin{equation}\label{eq: u eps is subsol}
{\rm MA}_{\theta}( \varphi_{\varepsilon}) \geq e^{\varphi_{\varepsilon}} (f+|f-g|)dV\geq e^{\varphi_{\varepsilon}}g dV.        
\end{equation}

It follows from  \eqref{eq: u eps is subsol}  that $\varphi_{\varepsilon}$ is a subsolution and $\psi$ is a supersolution of the equation ${\rm MA}_{\theta}(\phi)=e^{\phi}gdV$. The comparison principle \cite[Proposition 6.3]{Boucksom_Eyssidieux_Guedj_Zeriahi_2010_Big} insures that $\varphi_{\varepsilon}\leq \psi$, hence
\begin{eqnarray*}
        \varphi-\psi &= &\varphi_{\varepsilon}-\psi + \varepsilon (\varphi-\rho)  + C_2\varepsilon - n\log (1-\varepsilon)\\
        &\leq & (C_1 +C_2 + \sup_X (\varphi-V_{\theta}) + 2n) \varepsilon\\
        &= & (C_1+{\rm osc}_X (\varphi-V_{\theta})+2n) \exp{\left (\frac{\sup_X \varphi}{n}\right )} \|f-g\|_p^{1/n}.  
\end{eqnarray*}
The result follows since $C_1$ only depends on $p,\theta,n$. 
\end{proof}

 \end{document}